\documentclass[12pt,letterpaper,reqno]{amsart}

\usepackage{times} 
\usepackage[T1]{fontenc} 
\usepackage{mathrsfs} 
\usepackage{latexsym} 
\usepackage{float}
\usepackage{pdfpages}
\usepackage{graphicx}
\usepackage{blindtext}

\usepackage{epsfig}
\usepackage{amsmath,amsfonts,amsthm,amssymb,amscd} 
\input amssym.def 
\input amssym.tex

\newcommand\be{\begin{equation}} 
\newcommand\ee{\end{equation}}
\newcommand\bea{\begin{eqnarray}} 
\newcommand\eea{\end{eqnarray}} 
\newcommand\bi{\begin{itemize}}
\newcommand\ei{\end{itemize}} 
\newcommand\ben{\begin{enumerate}} 
\newcommand\een{\end{enumerate}}
\newcommand\bc{\begin{center}} 
\newcommand\ec{\end{center}} 
\newcommand\ba{\begin{array}} 
\newcommand\ea{\end{array}}

\newtheorem{thm}{Theorem}[section]

\newtheorem{prop}[thm]{Proposition}

\theoremstyle{definition} 

\begin{document}

\title[Can connected commuting graphs of finite groups have arbitrarily large diameter ?] 
{Can connected commuting graphs of finite groups have arbitrarily large diameter ?}


\author{Peter Hegarty} \address{Department of Mathematical Sciences, 
Chalmers University Of Technology and University of Gothenburg,
41296 Gothenburg, Sweden} \email{hegarty@chalmers.se}

\author{Dmitry Zhelezov} \address{Department of Mathematical Sciences, Chalmers 
University Of Technology and University of Gothenburg,
41296 Gothenburg, Sweden} \email{zhelezov@chalmers.se}

\subjclass[2000]{20P05 (primary).} \keywords{}

\date{\today}

\begin{abstract} 
We present a family of finite, non-abelian groups and propose that there are members of this family whose commuting graphs are connected and of arbitrarily large diameter. If true, this would disprove a conjecture of Iranmanesh and Jafarzadeh. While unable to prove our claim, we present a heuristic argument in favour of it. We also present the results of simulations which yielded explicit examples of groups whose commuting graphs have all possible diameters up to and including 10. Previously, no finite group whose commuting graph had diameter greater than 6 was known.
\end{abstract}


\maketitle

\setcounter{equation}{0}

\setcounter{equation}{0}

\setcounter{section}{-1}

\section{Notation}
Let $f,g : \mathbb{N} \rightarrow \mathbb{R}_{+}$. 
The following rather standard notations will be used in this \\ paper : 
\par (i) $f(n) \sim g(n)$ means that $\lim_{n \rightarrow \infty} \frac{f(n)}{g(n)} = 1$.
\par (ii) $f(n) \lesssim g(n)$ means that $\limsup_{n \rightarrow \infty} \frac{f(n)}{g(n)} \leq 1$.
\par (iii) $f(n) \gtrsim g(n)$ means that $g(n) \lesssim f(n)$.
\par (iv) $f(n) = O(g(n))$ means that $\limsup_{n \rightarrow \infty} \frac{f(n)}{g(n)} < \infty$.
\par (v) $f(n) = \Omega(g(n))$ means that $g(n) = O(f(n))$.
\par (vi) $f(n) = \Theta(g(n))$ means that both $f(n) = O(g(n))$ and $g(n) = O(f(n))$ hold.
\par (vii) $f(n) = o(g(n))$ means that $\lim_{n \rightarrow \infty} \frac{f(n)}{g(n)} = 0$. 
 
\section{Introduction}
Let $G$ be a non-abelian group. The {\em commuting graph} of $G$, denoted  
by $\Gamma(G)$, is usually defined as the graph whose vertices are the non-central elements of $G$, and such that $\{x,y\}$ is an edge if and only if 
$xy = yx$. Clearly, one can just as well define the graph to have as its vertices the non-identity cosets of $Z(G)$, such that $\{Zx,Zy\}$ is an edge if and only if $xy = yx$. This is the definition we will adopt henceforth. Commuting graphs of groups were first mentioned in the seminal paper of Brauer and Fowler \cite{BF} which was concerned with the classification of the finite simple groups. Interest in commuting graphs in their own right is often traced back to a question posed by Erd\H{o}s, and answered by Neumann \cite{N}, who showed that if the commuting graph of a group has no infinite independent set, then it cannot have arbitrarily large finite independent sets either. In this paper, we are only concerned with finite groups. \par If $\Gamma$ is any finite, connected graph, the diameter of $\Gamma$, denoted diam$(\Gamma)$, is defined to be the maximum of the distances between pairs of vertices in $\Gamma$. Here, the distance between vertices $x$ and $y$ is the minimum number of edges in a path from $x$ to $y$. If $\Gamma$ is disconnected, one sets diam$(\Gamma) := \infty$. The following interesting conjecture was made in \cite{IJ} :
\\
\\
{\bf Conjecture 1.1.}
{\em There is a natural number $b$ such that if $G$ is a finite, non-abelian group with $\Gamma(G)$ connected, then diam$(\Gamma(G)) \leq b$.}
\\
\\
This may seem like a very surprising conjecture at first sight, but in \cite{IJ} the authors provided some supporting evidence by proving that, for $n \geq 3$, the commuting graph of the full symmetric group $S_n$ is connected if and only if neither $n$ nor $n-1$ is prime, and in that case diam$(\Gamma(S_n)) \leq 5$. Previously, Segev and Seitz \cite{SeSe} had shown that if $G$ is a finite simple classical group over a field of size at least 5, then either $\Gamma(G)$ is disconnected or diam$(\Gamma(G)) \leq 10$ (it is not known if the upper bound is sharp). 
\par Conjecture 1.1 has attracted considerable attention, most recently with a preprint of Giudici and Pope \cite{GP}, which includes a comprehensive summary of existing partial results. Indeed, every paper published so far on this topic seems to take the viewpoint that Conjecture 1.1 is probably true and seeks to provide evidence to support it. Most of the evidence has come from groups which are structurally ``very far'' from being abelian, as in, for example, the results mentioned in the previous paragraph. The intuition here is clear: A priori, the more ``non-abelian'' a group is, the larger we would expect to be the diameter of its commuting graph. There is, of course, the risk that the graph becomes disconnected. But in many cases it turns out that the graphs are, in fact, 
connected, and when that is the case the diameter is bounded and small.
\par Giudici and Pope provide the first evidence in support of Conjecture 1.1 coming from $p$-groups. The following result of theirs is particularly striking:
\\
\\
{\bf Theorem 1.2 (Theorem 1.4 in \cite{GP}).} {\em If $G^{\prime} \subseteq Z(G)$ and $|Z(G)|^3 < |G|$, then \\ diam$(\Gamma(G)) = 2$.}  
\\
\\
The groups in this theorem are certainly ``very close to abelian'' in a structural sense, so that now we have evidence in support of Conjecture 1.1 coming from ``both extremes'', so to speak. 
\par Our purpose in this paper is to explain why we think Conjecture 1.1 is false. We will describe a family of groups which we believe violates the conjecture, and present some evidence that such is the case, though not a rigorous proof. Though we came upon our idea independently of Giudici and Pope, their work also gives strong hints where one might look for counterexamples. In addition to Theorem 1.2 above, they also proved (Theorem 1.5 in \cite{GP}) that if $(G:Z(G))$ is a product of three not necessarily distinct primes, then $\Gamma(G)$ is disconnected. Putting their two results together suggests one should look at groups of nilpotence class 2 in which the center is neither too large nor too small. This is exactly what we shall do. By the way, another natural reason to look at groups of nilpotence class 2 is that is it easier to ``keep track of'' the commutator relations in such groups, as the maps $g \mapsto [x,g]$ are additive, for every $x \in G$. 
\par The rest of the paper is organised as follows. Section 2 recalls some facts about the diameters of Erd\H{o}s-Renyi random graphs $G(n,p)$. It is known that, for any $\epsilon > 0$, if $p = n^{-1+\epsilon}$, then diam$(G(n,p))$ concentrates on $\lceil 1/\epsilon \rceil$. This material is standard, though may not be familiar to group theorists. Section 3 presents the family of groups whose commuting graphs are expected to achieve arbitrarily large diameter - our claim to this effect is formulated explicitly in Conjecture 3.4. The idea will be to construct a commuting graph randomly, so that it resembles $G(n,p)$, for the above value of the parameter $p$. We can arrange things so that $p$ has the desired value but, unlike in the Erd\H{o}s-Renyi graph, individual edges in our graph will not be chosen independently of one another. We will explain exactly what prevents one from transferring the analysis of $G(n,p)$ to our situation, while hoping to convince the reader that the former nevertheless provides a strong heuristic for Conjecture 3.4. Section 4 presents the results of simulations of our random group model which yielded explicit examples of commuting graphs of all possible diameters up to and including 10. Previously, computer searches of small groups yielded examples where the diameter was 6, and \cite{GP} also constructed an infinite family of groups where this is the case. The \texttt{Mathematica} code for our simulations is included in an appendix.   

\setcounter{equation}{0}

\section{The diameters of Erd\H{o}s-Renyi random graphs}

Let $n \in \mathbb{N}$ and $p = p(n) \in [0,1]$. Recall that $G(n,p)$ is the random graph on $n$ vertices in which every edge is present, independently of the others, with probability $p$. There is a great deal known about the diameter of $G(n,p)$, for different ranges of the parameter $p$ - see, for example, \cite{RW}. For our purposes, we are only interested in the case when $p$ is well above the connectivity threshold at $1/n$ and the diameter is concentrated on one value. The following fact is well-known:
\\
\\
{\bf Theorem 2.1.}
{\em Let $k\geq 2$ be an integer and let $\epsilon$ be a real number so that
$k-1 < 1/\epsilon < k$. Let $p = p(n) := n^{-1+\epsilon}$. Then, as $n \rightarrow 
\infty$, the random graph $G(n,p)$ is almost surely connected and of diameter
$k$.} 
\\
\\
This result seems to have been originally proven in \cite{KL}. We will sketch a
short proof using Janson's inequality. Given two vertices $a,b \in G(n,p)$ and a positive integer $l$, one may first estimate the expected number of paths of length $l$ between $a$ and $b$. The number of possible simple paths is clearly $\sim n^{l-1}$, and the probability of any particular such path being present is $p^l$. Hence, by linearity of expectation, the expected number 
of simple paths is $\sim n^{l-1} p^l$. Provided $p = \Omega(n^{-1})$, the same is true of the total expected number of paths, since the error term coming from non-simple paths is $\Theta(l) \cdot n^{l-2} p^{l-1}$. By the choice of $\epsilon$ and $p$ in the theorem, the expected number of paths behaves like a negative power of $n$ when $l < k$ and a positive power of $n$ when $l \geq k$. This is already enough to be able to deduce that the diameter of $G(n,p)$ is almost surely, as $n \rightarrow \infty$, greater than or equal to $k$. To prove equality, one needs to know that the number of paths from $a$ to $b$ is highly concentrated about its mean, and one can then employ a union bound over all possible choices of $a$ and $b$. Interestingly, a standard second moment analysis does not seem to provide a sufficiently strong concentration estimate here. So we instead use Janson's inequality, as stated in Theorems 8.1.1 and 8.1.2 of \cite{AS}. Paraphrasing, the inequality implies that the probability of there being no path from $a$ to $b$ satisfies
\be
\mathbb{P} ({\hbox{no path from $a$ to $b$}}) \leq \left\{ \begin{array}{lr} 
\exp \left( - \mu + \frac{\Delta}{2} \right), & {\hbox{if $\Delta \leq \mu$}}, 
\\
\exp \left( - \frac{\mu^2}{2\Delta} \right), & {\hbox{if $\Delta \geq \mu$}},
\end{array} \right.
\ee
where $\mu \sim n^{k-1}p^k = n^{k\epsilon - 1}$ is the expected number of paths, and the quantity $\Delta$ is defined as follows: Suppose the collection of all possible $a$-$b$ paths has been ordered in some way and let $B_{P_i}$ be the event that the $i$:th path is present. Then 
\be
\Delta = \sum_{i \sim j} \mathbb{P}(B_{P_i} \wedge B_{P_j}),
\ee
where the sum is taken over all ordered pairs $(i,j)$ for which the events $B_{P_i}$ and $B_{P_j}$ are dependent. Dependency just means that the paths $P_i$ and $P_j$ share at least one 
edge, and thus $\Delta$ can be easily estimated. Once again, the main term comes from pairs of simple paths, so it suffices to consider these. The number of edges common to a pair of dependent, simple paths of length $k$ can be any number $l \in \{1,2,...,k-1\}$. If there are $l$ common edges, then there are a total of $2k-l$ edges in the union of the two paths, and hence the probability that both paths are present is $p^{2k-l}$. It is also easy to see that the number of possible pairs in this case is $O_{k,l}(n^{2k-2-l})$, where the implicit constant takes account of the fact that we can choose where in each path of the pair the common edges are to be placed. Hence 
\be
\Delta = O_{k} \left( \sum_{l=1}^{k-1} n^{2k-2-l} p^{2k-l} \right) = (n^{k-1} p^k)^2 
\cdot O_{k} \left( \sum_{l=1}^{k-1} (np)^{-l} \right) = \mu^2 \cdot O_{k}(n^{-\epsilon}) = O_{k}(n^{2(k\epsilon - 1) - \epsilon}).
\ee
Since $k \epsilon - 1 < \epsilon$, for $n \gg 0$ we will certainly have $\Delta < \mu$, and hence the first alternative in (2.1) applies. It follows that, 
as $n \rightarrow \infty$,  
\be
\mathbb{P} ({\hbox{no path from $a$ to $b$}}) \leq \exp \left[ - (1+o_{n}(1))(n^{k\epsilon - 1}) \right].
\ee
By a simple union bound, the probability that every pair of vertices in $G(n,p)$ are connected by a path of length $k$ is $\geq 1 - n^{2} e^{-(1+o_{n}(1))(n^{k\epsilon - 1})}$, which goes to one as $n \rightarrow \infty$, thus proving Theorem 2.1.

\setcounter{equation}{0}

\section{The random group model}

Fix an integer $k \geq 2$ and a real number $\delta \in \left(0, \frac{1}{2k(k-1)} \right)$. Let $m$ be a positive integer and $n := 2^m - 1$. We will be letting $m \rightarrow \infty$ in due course. We can choose $\delta_1 > 0$ such that,
if $r := \lfloor (1-\delta_1) m \rfloor$ and $p := 2^{-r}$ then, for all $m \gg 0$, $1 + \log_{n} p \in \left( \frac{1}{k} + \delta, \frac{1}{k-1} - \delta \right)$.\par For all $m \gg_{k} 0$ we will define a random group $G_{m,k}$. With $r$ as above, let $Z(G_{m,k})$ be the elementary abelian $2$-group generated by elements $z_1,...,z_r,w_1,...,w_m$. The quotient $G_{m,k}/Z(G_{m,k})$ is to be the elementary abelian $2$-group generated by cosets $Zx_1,...,Zx_m$ such that $x_{i}^{2} = w_i$ for $i = 1,...,m$. So far there is nothing random in the definition of $G_{m,k}$. However, it remains to define the commutator relations $[x_i,x_j]$ for $1 \leq i < j \leq m$, and this is where the randomness will be introduced. Namely, for each pair $\{i,j\}$, independent of the others, we choose $[x_i,x_j]$ to be a uniformly random element of the subgroup $H$ 
of $Z(G_{m,k})$ generated by $z_1,...,z_r$. Some basic properties of the random group $G_{m,k}$ are straightforward to prove:

\begin{prop}
For any two distinct, non-empty subsets $\mathcal{S}$ and $\mathcal{T}$ of $\{1,...,m\}$ we have
\be
\mathbb{P} \left( \left[ \prod_{s \in \mathcal{S}} x_s, \prod_{t \in \mathcal{T}} x_t \right] = 1 \right) = 2^{-r}.
\ee
\end{prop}

\begin{proof}
The result is true by definition if $\mathcal{S}$ and $\mathcal{T}$ are singleton sets. It is therefore true in general since we can just make a change of variables, that is, a linear automorphism of $<Zx_1,...,Zx_m>$, considered as a vector space of dimension $m$ over $\mathbb{F}_2$. This principle, that the probabilities of certain events defined by relations between $x_1,...,x_m$ are invariant under a change of variables, will be used repeatedly in the remainder of this section. We give this principle a name, the {\em linear invariance principle (LIP)}. Let us give a rigorous proof that the principle indeed holds in the setting of the current lemma - in future, we will just refer to the principle without proof, as in all our applications of it it will be equally obvious that the application is valid.  
\par So let $\mathcal{S}$ and $\mathcal{T}$ be distinct 
non-empty subsets of $\{1,...,m\}$. Let $\mathcal{U} := \mathcal{S} \cap \mathcal{T}$ and $\mathcal{V} := (\mathcal{S} \times \mathcal{T}) \backslash (\mathcal{U} \times \mathcal{U})$. That $\mathcal{S}$ and $\mathcal{T}$ are distinct implies that $\mathcal{V}$ is non-empty. Then 
\be
 \left[ \prod_{s \in \mathcal{S}} x_s, \prod_{t \in \mathcal{T}} x_t \right] = 
\prod_{(i,j) \in \mathcal{V}} [x_i,x_j].
\ee
For each $(i,j) \in \mathcal{V}$, let $X_{(i,j)}$ denote the random variable $[x_i,x_j]$, which is uniformly distributed over the subgroup $H$. Let $X_{(\mathcal{S},\mathcal{T})}$ denote the random variable on the left-hand side of (3.2). Thus, considering $H$ as a vector space over $\mathbb{F}_2$, we have  
\be
X_{(\mathcal{S},\mathcal{T})} = \sum_{(i,j) \in \mathcal{V}} X_{(i,j)}.
\ee 
Note that if $(i,j) \in \mathcal{V}$ then $(j,i) \not\in \mathcal{V}$. Hence 
$X_{(\mathcal{S},\mathcal{T})}$ is a sum of $|\mathcal{V}|$ independent variables, each uniformly distributed over $H$. It follows that $X_{(\mathcal{S},\mathcal{T})}$ is also uniformly distributed over $H$, and this implies (3.1).
\end{proof}
 
\begin{prop}
As $m \rightarrow \infty$, $\mathbb{P}(G_{m,k}^{\prime} = \; <z_1,...,z_r>) \rightarrow 1$.
\end{prop}

\begin{proof}
First of all note that, since the squares $x_{i}^{2}$, $i = 1,...,m$, are automatically linearly independent of all the commutators $[x_i,x_j]$, whatever values we assign to the latter, we cannot introduce any contradictions. It thus remains to show that, with high probability (w.h.p.), the commutators $[x_i,x_j]$
span the subgroup $H$ generated by $z_1,...,z_r$, which we continue to regard as a vector space over $\mathbb{F}_2$. There are $2^r - 1$ codimension-one subspaces of $H$. List them in any order, and for each $\xi = 1,...,2^r - 1$, let 
$B_{\xi}$ be the event that all the commutators $[x_i,x_j]$ lie in the $\xi$:th subspace. LIP implies that each of the events $B_{\xi}$ have the same probability. Hence, by a simple union bound, 
\be
\mathbb{P} \left( \bigcup_{\xi = 1}^{2^r - 1} B_{\xi} \right) \leq \sum_{\xi = 1}^{2^r - 1} 
\mathbb{P}(B_{\xi}) = (2^r - 1) \mathbb{P}(B_1).
\ee
Without loss of generality, we may assume that the first codimension-one subspace in our ordering is that spanned by $z_1,...,z_{r-1}$. Since the value of any commutator $[x_i,x_j]$ is chosen at random, the probability of it lying inside this subspace is just $1/2$. Since the choices for different pairs are independent, it follows that 
\be
\mathbb{P}(B_1) = 2^{-C(m,2)} \leq 2^{-C(r,2)}.
\ee
From (3.4) and (3.5), it follows that the commutators span $H$ w.h.p. as $m$, and hence $r$, tend to infinity.
\end{proof}

\begin{prop}
With high probability, $Z(G_{m,k})$ is spanned by $z_1,...,z_r,w_1,...,w_m$ and nothing else. In other words, with high probability, there is no non-empty subset $\mathcal{S} \subseteq \{1,...,m\}$ such that 
\be
\left[ \prod_{s \in \mathcal{S}} x_s, x_i \right] = 1, \;\; {\hbox{for all $i = 1,...,m$}}.
\ee
\end{prop}

\begin{proof}
For each non-empty subset $\mathcal{S} \subseteq \{1,...,m\}$, let $B_{\mathcal{S}}$ be the event that (3.6) holds. By LIP, each of the events $B_{\mathcal{S}}$ have the same probability. Since there are $2^m - 1$ such events it suffices, by a union bound, to show that $\mathbb{P}(B_{\{1\}}) = o(2^{-m})$. But $B_{\{1\}}$ occurs if and only if $x_1$ commutes with each of $x_2,...,x_m$. These $m-1$ events are independent, and each occurs with probability $2^{-r}$. Hence $\mathbb{P}(B_{\{1\}}) = 2^{-(m-1)r}$, which is $o(2^{-m})$, since $r = \Theta(m)$, as desired. 
\end{proof}     

Propositions 3.2 and 3.3 imply that, as $m \rightarrow \infty$, the group $G_{m,k}$ almost surely has the following three properties:
\par (I) it is nilpotent of class $2$,
\par (II) it is of order $2^{2m+r}$, 
\par (III) $G_{m,k}^{\prime}$, $Z(G_{m,k})$ and $G_{m,k}/Z(G_{m,k})$ are all elementary abelian, of orders $2^r, 2^{m+r}$ and $2^m$ respectively. 
\\
Hence the commuting graph $\Gamma(G_{m,k})$ almost surely consists of $n = 2^m - 1$ nodes, one for each non-empty subset $\mathcal{S} \subseteq \{1,...,m\}$. By Proposition 3.1, each edge of this graph is present with probability $p = 2^{-r}$. As explained at the outset, our choice of parameters ensures that $p = p(n) = n^{-1 + \epsilon_n}$, where $\epsilon_n \in \left( \frac{1}{k} + \delta, \frac{1}{k-1} - \delta \right)$. Hence, by analogy with Erd\H{o}s-Renyi graphs, we expect that the following holds:
\\
\\
{\bf Conjecture 3.4.} {\em As $m \rightarrow \infty$, $\Gamma(G_{m,k})$ is almost surely connected and has diameter $k$.}
\\
\\
We have been unable to prove this assertion, but we would be amazed if at least the first part of it, namely the claim that $\Gamma(G_{m,k})$ is a.s. connected, were false. Note that even that much would suffice to disprove Conjecture 1.1 since, as we will see below, it is easy to prove that the diameter of $\Gamma(G_{m,k})$ is a.s. at least $k$. The obvious line of attack for Conjecture 3.4 is to try to imitate, and modify where necessary, the proof given earlier of Theorem 2.1. Obviously, some modification is necessary since, unlike in $G(n,p)$, the edges of $\Gamma(G_{m,k})$ are not chosen independently of one another. Indeed, the graph is completely specified by the states of the $C(m,2)$ edges $\{Zx_i,Zx_j\}$, $1 \leq i < j \leq m$. On the one hand, we claim that first and second moment analyses concerning the number of paths between a pair of vertices of $\Gamma(G_{m,k})$ can be carried out as in the Erd\H{o}s-Renyi setting, with only minor technical modifications. To show this rigorously requires some work, which we now perform.
\par Fix a positive integer $l$. Let 
$a,b$ be two distinct vertices of $\Gamma(G_{m,k})$. Let 
\be
P : \gamma_0 = a \rightarrow \gamma_1 \rightarrow \gamma_2 \rightarrow \cdots \rightarrow \gamma_{l-1} \rightarrow b = \gamma_l
\ee
be a path of length $l$ between $a$ and $b$ in the complete graph $K_n$ on $n$ vertices. Let $B_P$ denote the event that $P$ is present in $\Gamma(G_{m,k})$. Further, 
let $H_P$ denote the subspace of $G_{m,k}/Z(G_{m,k})$ spanned by $a,\gamma_1,...,\gamma_{l-1},b$. Since $a \neq b$ we have a priori that
\be
2 \leq {\hbox{dim}}(H_P) \leq l+1.
\ee
For each $t \in \{2,...,l+1\}$, let $E_{l,t}$ denote the number of paths $P$ of 
length $l$ between $a$ and $b$ in $K_n$ for which ${\hbox{dim}}(H_P) = t$. Note that this number does not depend on the choice of $a$ and $b$.
If $P_1$ and $P_2$ are two paths between $a$ and $b$, set 
$H_{P_1,P_2} :=$ \\ $<H_{P_1},H_{P_2}>$. A priori we have 
\be
2 \leq {\hbox{dim}}(H_{P_1,P_2}) \leq 2l.
\ee
Finally, for each $t \in \{2,...,2l\}$, let $F_{l,t}$ denote the number of ordered pairs $(P_1,P_2)$ of paths of  
length $l$ between $a$ and $b$ in $K_n$ for which ${\hbox{dim}}(H_{P_1,P_2}) = t$.
All the crucial facts we need are contained in the next proposition:
\\
\\
{\bf Proposition 3.5.}
{\em With notation as above, and letting $m,n \rightarrow \infty$ for a fixed $l$, we have the following :
\par 
(i) $\mathbb{P}(B_P) \leq p^{{\hbox{dim}}(H_P)-1}$, and we have equality when 
${\hbox{dim}}(H_P) = l+1$.
\par (ii) $E_{l,t} = \Theta_l(1) \cdot n^{t-2}.$
\par (iii) The events $B_{P_1}$ and $B_{P_2}$ are independent when ${\hbox{dim}}(H_{P_1,P_2}) = 2l$.
\par (iv) For any two paths, $\mathbb{P}(B_{P_1} \wedge B_{P_2}) 
\leq p^{{\hbox{dim}}(H_{P_1,P_2})}$, with equality when the dimension is $2l$.
\par (v) $F_{l,t} = \Theta_l(1) \cdot n^{t-2}$.}

\begin{proof}
(i) Consider a path as in (3.7). The point is that, as we run through the 
vertices $\gamma_i$, each time the addition of a $\gamma_i$ increases the 
dimension of the space $H_P$ by one, the edge corresponding to the commutator $[\gamma_i,\gamma_{i-1}]$ is independent of all the previous edges along the path. This is an immediate consequence of our definition of $G_{m,k}$, along with an application of LIP if one wants to be completely rigorous. The inequality for 
$\mathbb{P}(B_P)$ follows immediately in turn. When the dimension of $H_P$ is maximal then all the edges along the path are independent, which gives equality in that case.
\\
\\
(ii) Consider paths as in (3.7) again. We must estimate the number of ways 
we can choose the ordered sequence $(\gamma_1,...,\gamma_{l-1})$ of vertices, assuming that dim$(H_P) = t$. We can just as well start with $a$ and $b$, which span a $2$-dimensional space, and choose the remaining $\gamma_i$ in order, so that the dimension increases by $t-2$ in all. Each time we choose a new $\gamma_i$ we must decide whether to increase the dimension by one or not. Each time we do the former, there are $(1-o_{n}(1))n$ choices for $\gamma_i$. Each time we do the latter, there are certainly no more than $2^{l} = O_l(1)$ choices for $\gamma_i$. We will get another $O_l(1)$ factor from the freedom to choose on which $t-2$ occasions the dimension is to be increased. But clearly the result is that
we have $\Theta_l(1) \cdot n^{t-2}$ choices for the path, as claimed. 
\\
\\
The argument for (iii) and (iv) follows the same lines as that for (i), while the argument for (v) follows the same lines as that for (ii). We omit further details.    
\end{proof}

Because of Proposition 3.5, parts of the proof of Theorem 2.1 can be run through more or less verbatim. Firstly, it follows from parts (i) and (ii) of the proposition that, in estimating the expected number $\mu$ of paths of a fixed length $l$ in $G_{m,k}$ between a pair of vertices $a$ and $b$, we will have a main term coming from paths $P$ for which dim$(H_P) = l+1$, and an error term coming from all other paths. The main term is $\Theta_l(1) \cdot n^{l-1} p^{l}$, and the error term is $\Theta_l(1) \cdot n^{l-2} p^{l-1}$. This is exactly as in the proof of Theorem 2.1, up to $\Theta_l(1)$ factors, which make no difference for fixed $l$ as $m,r,n \rightarrow \infty$. Similarly, in the computation of $\Delta$, it follows from parts (iii), (iv) and (v) of the proposition that we will get a sum just like (2.3), except that the implicit constant depending on $k$ will be different. Once again, this makes no difference as $m,r,n \rightarrow \infty$, since $k$ is fixed. 
\par Since we can estimate $\mu$ and $\Delta$ more or less as before, we can perform a second moment analysis. The problem, as already stated in Section 2, is that this won't yield a.s. connectedness of the graph. Indeed, a standard application of the second moment method (see Chapter 4 of \cite{AS}) gives that, for any pair $a,b$ of vertices in $\Gamma(G_{m,k})$, 
\be
\mathbb{P}({\hbox{no path from $a$ to $b$}}) \leq \frac{\Delta + \mu}{\mu^2} = O(n^{-\delta_2}),
\ee
where $\delta_2$ is some positive number, depending on the choice of the parameters $\delta$ and $\delta_1$ at the outset of Section 3. From (3.10) and a simple averaging argument we can deduce
\\
\\
{\bf Proposition 3.6.} {\em There is some $\delta_3 > 0$, depending on the choices of $\delta$ and $\delta_1$, such that, as $m \rightarrow \infty$, $\Gamma(G_{m,k})$ a.s. has a connected component of size at least $n - n^{1-\delta_3}$.}
\\
\\
But to prove a.s. connectedness of the full graph we need much stronger concentration of the number of $a$-$b$ paths than that provided by (3.10). For Theorem 2.1 we applied Janson's inequality, but we cannot do that here. Janson's inequality assumes that there is an underlying set of independent coin tosses such that each event $B_P$ represents the success of a certain subset of these tosses. That requirement is satisfied by $G(n,p)$ but not $\Gamma(G_{m,k})$, since the edges in the latter graph are not placed independently. At this time we do not see how to get around this problem, so Conjecture 3.4 remains open. However, we hope the reader will agree with us that the theoretical evidence in its favour seems strong. The next section will also provide numerical evidence.  

\setcounter{equation}{0}

\section{Simulations}

We wrote a \texttt{Mathematica} program which takes as input integers $m$ and $r$ and creates a random group $G$ which, with high probability, has properties (I), (II) and (III) in Section 3. This was achieved by generating a random binary string of length $r$ for each of the $C(m,2)$ commutators $[x_i,x_j]$. In other words, in the notation of Section 3, for each pair $\{i,j\}$ with $1 \leq i < j \leq m$, we generate a binary string $\epsilon_1 \cdots \epsilon_r$, $\epsilon_i \in \{0,1\}$, and set $[x_i,x_j] := \prod_{i=1}^{r} z_{i}^{\epsilon_i}$. Our program then computes the diameter of the resulting commuting graph $\Gamma(G)$. Note that, if the diameter is at least $3$, then properties (I)-(III) definitely hold, though if the diameter is $2$, then there could be extra elements in $Z(G)$. The full source code for our program is reproduced in Figure 1 on page 11 and may also be obtained directly from the authors.  
\par The first table below presents the results of simulations for $m=7$ and different values of $r$. For each $r$, we generated 1000 graphs and the table shows the distribution of their diameters.

\begin{table}[ht!]
\begin{center}
\begin{tabular}{|c|c|} \hline
$r$ & ${\hbox{Number of graphs of each diameter}}$ \\ \hline \hline
$3$ & $\mathbf{2:}1000$        \\ \hline
$4$ & $\mathbf{3:}8 \;\;\; \mathbf{4:}977 \;\;\; \mathbf{5:}15$ \\ \hline
$5$ & $\mathbf{5:}14 \;\; \mathbf{6:}97 \;\; \mathbf{7:}79 \;\; \mathbf{8:}8 \;\; \mathbf{9:}1 \;\; \mathbf{\infty:}801$ \\ \hline
$6$ & $\mathbf{\infty:} 1000$ \\ \hline
\end{tabular}
\end{center}
\end{table}  

In this table, as $r$ increases, the diameters of the graphs generally increase, as expected. Indeed, since the commutators $[x_i,x_j]$ are determined by generating independent bits, it is easy to prove (by a coupling argument) that the expected value of the diameter is an increasing function of $r$. For $r=3$ every graph had diameter $2$ and for $r=4$ nearly 98$\%$ of graphs had diameter $4$. Hence in these cases, the diameter is strongly concentrated on one value. Note, however, that for $r=4$ this value is bigger than what Conjecture 3.4 would predict, something which may or may not simply be due to the small value of $m$. For $r=5$, about 80$\%$ of graphs were disconnected, and the diameters of those which were connected show a greater spread. Here we already have examples of graphs of all diameters up to $9$. 
\par The second table below shows the output of $200$ simulations for $m=8$:

\begin{table}[ht!]
\begin{center}
\begin{tabular}{|c|c|} \hline
$r$ & ${\hbox{Number of graphs of each diameter}}$ \\ \hline \hline
$4$ & $\mathbf{3:}200$        \\ \hline
$5$ & $\mathbf{4:}38 \;\;\; \mathbf{5:}162$ \\ \hline
$6$ & $\mathbf{8:}4 \;\;\; \mathbf{\infty:}196$ \\ \hline
\end{tabular}
\end{center}
\end{table}  

For $r=6$ we allowed our program to keep running and it subsequently also generated graphs of diameters 6,7 and 9 and, finally, after several hours and on the 716:th simulation, a graph of diameter 10. The commutator relations for this group are reproduced in full in Figure 2 at the end of the paper. Of the 716 graphs sampled for $r=6$, a total of 22 were connected.  
\par In the above data we thus see the following pattern. For fixed $m$ and sufficiently small values of $r$, the graph $\Gamma(G)$ is always connected and its diameter is highly concentrated about one value. There is a minimum value $r_0(m)$ where disconnected graphs begin to appear, and the probability that $\Gamma(G)$ is disconnected is already very high at $r = r_0(m)$. Those ``outliers'' which are connected have a much wider spread of diameters. Though as $m$ increases it rapidly becomes impractical to collect significant amounts of data, we thus have some very tentative evidence for an extended version of Conjecture 3.4, to the effect that there should be a sharp connectivity threshold for our random commuting graph model, similar to the well-known situation for $G(n,p)$, though perhaps for a different threshold function $p = p(n)${\footnote{Note that, in our model, the parameter $p$ can only take on a discrete set of values, namely negative powers of 2. This does not affect our ability to precisely define a {\em threshold function} $p(n)$, however.}}. The difficulty of obtaining data as $m$ grows is exhibited by the fact that, for the pair $(m,r) = (12,9)$, it took about a week for our program to produce $48$ graphs, of which $38$ had diameter $7$ and the remainder had diameter $6$. This is the smallest value of $m$ for which we could reliably generate connected graphs of greater diameter than were previously known in the literature. We have not observed any graph of diameter greater than $10$, but have every reason to believe that they exist. Note that, since the complexity of computing the diameter of a graph grows, in general, as the cube of the number of vertices, the running time of any program which simulates $G_{m,k}$ is of the order of $2^{3m}$.


\begin{figure*}[ht!]
  \includegraphics[page=1,width=\textwidth]{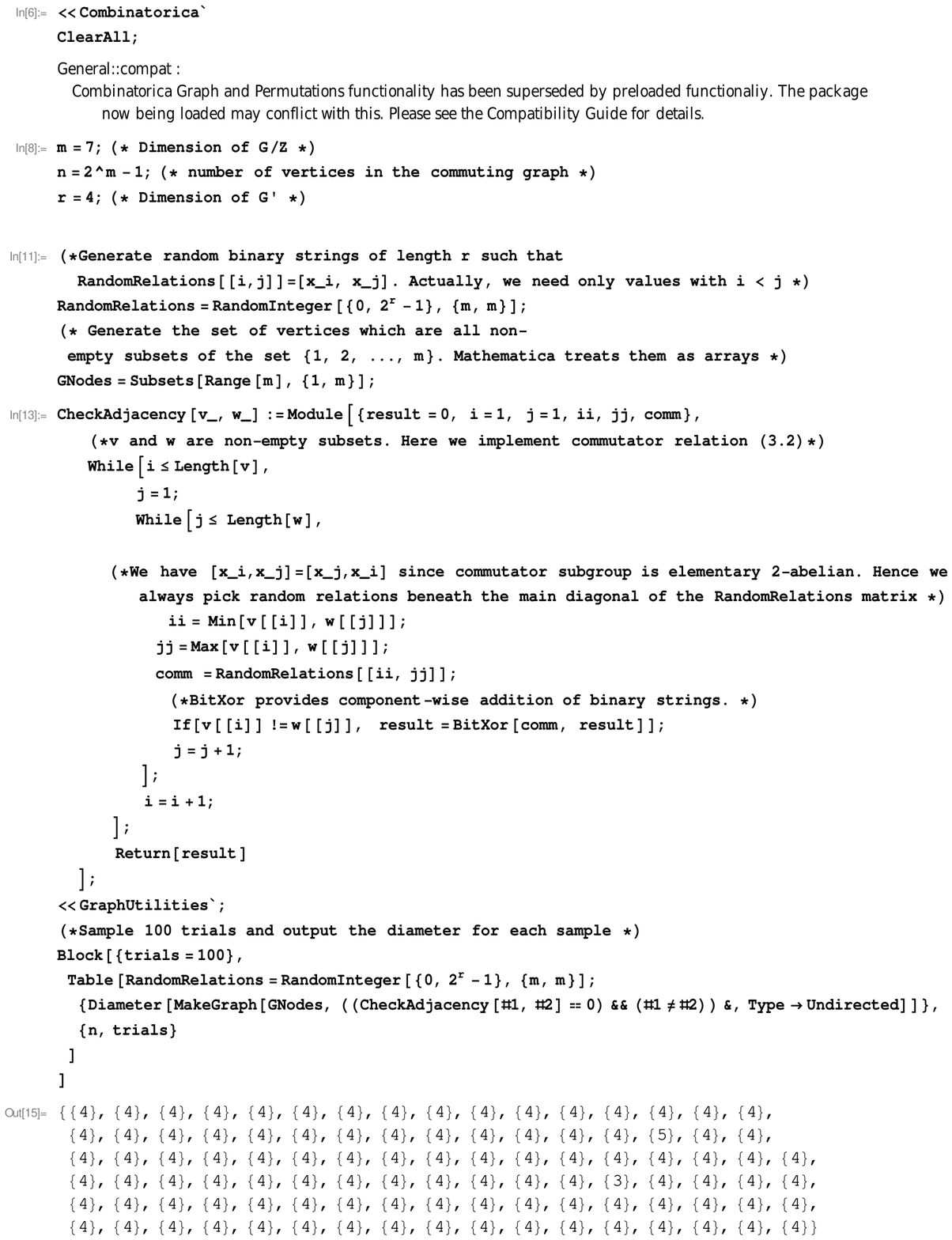} 
 \label{fig:ini}
\caption{Source code for $m=7$, $r=4$ and $100$ simulations.}
\end{figure*}

\begin{figure*}[ht!]
       \includegraphics[width=\textwidth]{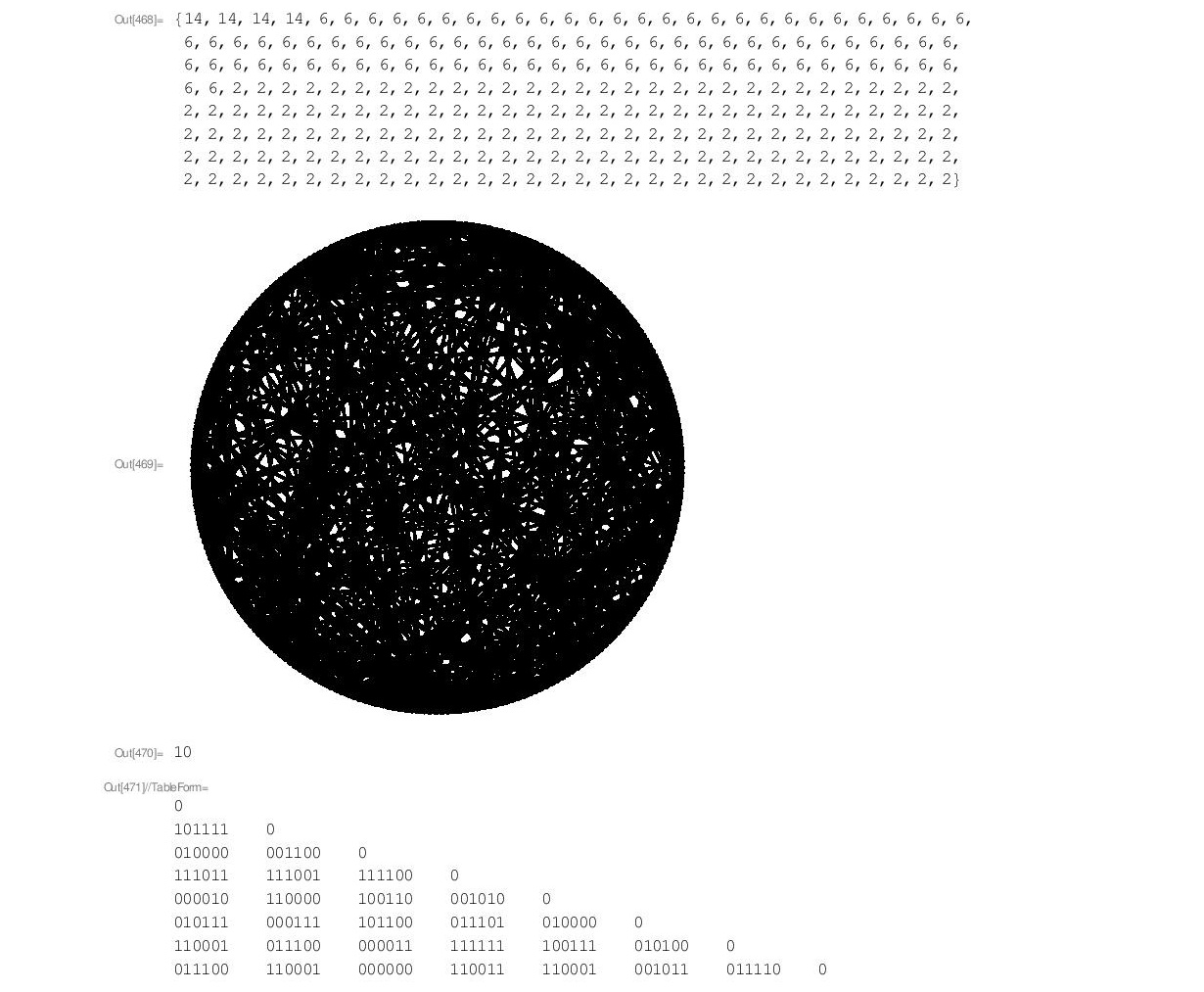} 
 \label{fig:typical}
\caption{The degree sequence and commutator relations for a graph of diameter $10$, with $m=8$, $r=6$.} 
\end{figure*}


\end{document}